\documentclass{amsart}
\usepackage{graphicx}
\newtheorem{theorem}{Theorem}[section]
\newtheorem{lemma}[theorem]{Lemma}

\theoremstyle{definition}

\theoremstyle{remark}
\newtheorem{remark}[theorem]{Remark}

\numberwithin{equation}{section}

\begin{document}

\title{A priori error analysis of the $hp$-mortar FEM for parabolic problems}
\author{Sanjib Kumar Acharya} 
\address{TIFR Centre for Applicable Mathematics, Bengaluru, Karnataka 560065 India} 
\email{acharya.k.sanjib@gmail.com}
\author{Ajit Patel}
\address{The LNM Institute of Information Technology,  Jaipur, Rajasthan 302031  India } 
\email{ajit.iitb@gmail.com}
\author{Talal Rahman} 
\address{Western Norway University of Applied Sciences, Department of Computing, Math\-ematics and Physics,
Inndalsveien 28, 5063 Bergen, Norway} 
\email{Talal.Rahman@hvl.no}

%





\begin{abstract}
In this article we derive a priori error estimates for the $hp$-version of the mortar finite element method for parabolic initial-boundary value problems.  Both semidiscrete
and fully discrete  methods are analysed in $L^2$- and $H^1$-norms.  The superconvergence results for the solution of the semidiscrete problem  are studied in an eqivalent negative norm,  with an extra regularity assumption. Numerical experiments are conducted to validate the theoretical findings.
\end{abstract}

\maketitle

\section{Introduction}
 Over the last two decades, mortar finite element methods have attracted
plenty of attentions due to its intriguing features like, flexibility in handling different
types of nonconformities and various complex or even unsteady geometries. It is a domain
decomposition method which allows to divide the domain of definition into several overlapping
or nonoverlapping subdomains and to choose independent discretization scheme in different subdomains. The grids are glued together by a mortar projection without disturbing the local discretization.
This method is very successful in  approximating the solution in  nonhomogeneous mediums, cf. \cite{bernardy1}.
It needs the local behavior of the exact solution of the partial differential equation which must be
approximated. Nonconformity in the method is either caused by  the choice of the matching conditions
of the discrete solution on the interfaces or by the geometrical features of the partition of the domain.
In standard mortar finite element method, the variational formulation leads to positive definite system on
the constrained mortar space, cf. \cite{bernardy}. On the other hand, a Lagrange multiplier is used to alleviate the mortaring condition and the variational formulation gives rise to an indefinite system, cf. \cite{belagacem-h}.

Convergence of the finite element methods can be achieved in three ways, by decreasing the discretization parameter $h$ (the $h$-version) or by increasing the polynomial degree $p$  (the $p$-version) or by combining the both (the $hp$-version). For the standard $h$-, $p$- and $hp$-version of the  finite element methods for elliptic and parabolic problems, we refer to \cite{babuska, babuska1, babuska-p, janik-p, ciar, thomee}. For a general discussion on $p$- and $hp$-version of the finite element method we refer to \cite{general-hp}.

Optimal error estimates for the $h$-version of the mortar finite element method with and without
Lagrange multiplier for elliptic problems are established in \cite{belagacem-h,bernardy}.
For a general discussion on mortar element methods and its application we refer to \cite{bernardy1}
and the references therein. The $h$-version of the mortar finite element method with and without
Lagrange multiplier for the parabolic initial boundary value problems is introduced  by Patel et al.
\cite{patel} in which optimal error estimates are established. We refer to \cite{overlap,overlap1} for mortar element method
with overlapping partition.

The $hp$-version of the mortar finite element methods with and without Lagrange
multiplier (with meshes satisfying a generalized condition) for elliptic problems
are introduced by Seshaiyer and Suri (cf. \cite{suri}). Wherein,  suboptimal error
estimate  with a pollution term $p^{3/4}$ has been obtained for quasiuniform meshes.
However, significant deterioration of the accuracy is not seen in the computational
results compared to the optimal rate (cf. \cite{sesh2}). Pollution in the error estimate
is caused due to the continuity constant of the non-quasiuniform mortar projection operator and
it can not be improved as it is observed  and  computationally verified by the eigenvalue
technique  (see  \cite{sesh1, suri,sesh3}). Some variants of the mortar finite element method
(M1, M2 methods) are introduced by Seshaiyer and Suri in \cite{sesh2} in which suboptimal error estimates
are established.  A new variant (MP method) is introduced by Belgacem et al. in \cite{belgacem-hp}
where they derived suboptimal error estimates. We refer to \cite{belgacem-hp,chilton} for $hp$-version of  the mortar finite element
method for fluid flow problems.

In \cite{belgacem-hp},  improved estimates are derived using an interpolating argument, which are suboptimal with a pollution term $O(p^{\epsilon})$. This technique may fail in some interesting situations (cf. \cite{sesh}).
The loss term is  reduced to $O(\sqrt{\log p})$ by Belgacem et al. in \cite{sesh1},
and quasi optimal results are established for mixed elasticity and  Stokes problems.

For problems with singularities, the finite element solution with quasiuniform mesh behaves very harshly near the singular points (cf. \cite{babuska,babuska1}). By employing geometric meshes
it is seen that the error decays exponentially even in the presence of those singularities. For the standard
$hp$ estimates with non-quasiuniform meshes, we refer to \cite{general-hp,Guo}.  Suitable meshing in
the vicinity of singular points gives an exponential decay of the error in $hp$ version of the the
mortar finite element method irrespective of the pollution term (cf. \cite{suri,sesh2}).

In the last decade, a number of articles were published concerning the error estimates for the  $hp$-version of the mortar finite element methods for elliptic problems. But till date there is hardly any article available for the parabolic problems.
In this article our aim is to establish a priori error estimates  for the parabolic problems (with quasiuniform mesh over subdomains) in $L^2$- and $H^1$-norms. There are situations where it is essential to use negative norm estimates of the solution to get superconvergence. We derive the  estimates of the mortar solution in an equivalent negative norm which give superconvergence of the method with an extra regularity assumption. To make it concise we are not dealing with the results for the nonquasiuniform mesh which can be derived similarly using the results in \cite{suri}.

The rest of the paper as follows. In Section 2, we  briefly recall function spaces to be used
in this manuscript. We formulate a parabolic initial boundary value problem in the context of mortar
finite element method in Section 3. In section 4, we state and develop some approximation properties which plays
a vital role in proving the convergence results. Error estimates for semidiscrete scheme are discussed in Section 5. We discuss the superconvergence of the method in Section 6. Fully discrete
scheme is discussed in Section 7 in which quasioptimal results are obtained.  Finally in Section 8, we give
concluding remarks.

\section{Preliminaries}
We define the space $L^{2}(0,T;Y)$  (cf. \cite{lion}) as
\begin{equation*}
L^{2}(0,T;Y)=\left\{v:[0,T]\rightarrow Y:\int_{0}^{T}{||v||}^{2}_{Y}~dt<\infty\right\}
\end{equation*}
equipped with the following norm
$${||v||}_{L^{2}(0,T;Y)}=\left(\int_{0}^{T}{||v||}^{2}_{Y}~dt\right)^{1/2},$$
where $Y$ is a Banach space and $0<T<\infty$.

Let $\Omega$ be an open bounded polygonal domain in  $\mathbb{R}^{2}$  with boundary $\partial\Omega$.
We denote $\bar{\alpha}=(\alpha_{1},\alpha_{2})$ as a $2$-tuple of non-negative integers $\alpha_{i}, \ i=1, 2,$
with $|\bar\alpha|=\alpha_{1}+\alpha_{2}$, and set
$$D^{\bar\alpha}=\frac{\partial^{|\bar\alpha|}}{\partial x_{1}^{\alpha_{1}}\partial x_{2}^{\alpha_{2}}}\cdot$$
The Sobolev space of integer order $m$, over the domain $\Omega$, cf. \cite{grisvard}, is defined as follows,
\begin{equation*}
H^{m}(\Omega)=\left\{v\in L^{2}(\Omega):D^{\bar\alpha}v\in L^{2}(\Omega), |\bar\alpha|\leq m\right\},
\end{equation*}
and is equipped with the norm and semi-norm defined as follows,
\begin{equation*}
{||v||}_{H^{m}(\Omega)}=\left(\sum_{|\bar\alpha|\leq m}\int_{\Omega}|D^{\bar\alpha}v|^{2}dx\right)^{1/2}\text{and}~~
{|v|}_{H^{m}(\Omega)}=\left(\sum_{|\bar\alpha|= m}\int_{\Omega}|D^{\bar\alpha}v|^{2}dx\right)^{1/2},
\end{equation*}
respectively. 
We define a negative norm ${||\cdot||}_{H^{-m}(\Omega)}$ by
$${||v||}_{H^{-m}(\Omega)}=\sup_{0\neq\phi\in H^{m}(\Omega)}\frac{(v,\phi)}{{||\phi||}_{H^{m}(\Omega)}},$$
where $v\in L^{2}(\Omega)$ and $(\cdot,\cdot)$ denotes the usual inner product in $L^{2}(\Omega)$.

Let $\nu=m+\sigma$ be a positive real number, where $m$  and  $\sigma\in(0,1)$ are  the integral part
and fractional part of $\nu$, respectively. The fractional order Sobolev space $H^{\nu}(\Omega)$ is defined as
\begin{equation*}
H^{\nu}(\Omega)=\bigg\{v\in H^{m}(\Omega):\int_{\Omega}\int_{\Omega}\frac{\big(D^{\bar\alpha}v(x)
-D^{\bar\alpha}v(y)\big)^{2}}{|x-y|^{2+2\sigma}}dxdy<\infty,~|\bar\alpha|=m\bigg\}
\end{equation*}
with the norm
\begin{equation}\nonumber
{||v||}_{H^{\nu}(\Omega)}=\bigg({||v||}^{2}_{H^{m}(\Omega)}
+\sum_{|\bar\alpha|=m}\int_{\Omega}\int_{\Omega}\frac{\big(D^{\bar\alpha}v(x)
-D^{\bar\alpha}v(y)\big)^{2}}{|x-y|^{2+2\sigma}}dxdy\bigg)^{1/2}.
\end{equation}

We shall denote by $H^{\nu-1/2}(\partial\Omega)$ the space of traces ${v|}_{\partial\Omega}$
over $\partial\Omega$ of the functions $v\in H^{\nu}(\Omega)$ equipped with the norm
$${||g||}_{H^{\nu-1/2}(\partial\Omega)}=\inf_{v\in H^{\nu}(\Omega),v{|}_{\partial\Omega}
=g}{||v||}_{H^{\nu}(\Omega)}$$ and $$H^{1}_{0}(\Omega)=\{v\in H^{1}(\Omega):v|_{\partial\Omega}=0\}.$$

For $\sigma\in(0,1)$ let $H^{-\sigma}(\partial\Omega)$ be the dual space of $H^{\sigma}(\partial\Omega)$,
equipped with the norm $${||\mu||}_{H^{-\sigma}(\partial\Omega)}=\sup_{g\in H^{\sigma}(\partial\Omega),~g\neq 0}\frac{\langle\mu,g\rangle_{\sigma,\partial\Omega}}{||g||_{H^{\sigma}(\partial\Omega)}},$$
where $\langle\cdot,\cdot\rangle_{\sigma,\partial\Omega}$ is the duality pairing between $H^{-\sigma}(\partial\Omega)$
and $H^{\sigma}(\partial\Omega)$.

 Let $\gamma$ be a part of $\partial\Omega$ and we define $H^{1/2}_{00}(\gamma)$ as an interpolation
space (cf. \cite{lof}) in between $L^{2}(\gamma)$ and $H^{1}_{0}(\gamma)$ that is,
\begin{equation}
H^{1/2}_{00}(\gamma)=[L^{2}(\gamma),H^{1}_{0}(\gamma)]_{1/2}
\end{equation}
endowed with the norm
\begin{equation}
{||g||}_{H^{1/2}_{00}(\gamma)}=\inf_{v\in H^{1}_{0}(\gamma),~v|_{\gamma}=g}{||v||}_{H^{1}(\Omega)}.
\end{equation}
We denote the dual space of $H^{1/2}_{00}(\gamma)$  by $H^{-1/2}_{00}(\gamma)$ together  with the norm
$${||\mu||}_{H_{00}^{-1/2}(\gamma)}=\sup_{g\in H^{1/2}_{00}(\gamma),~g\neq 0}\frac{\langle\mu,g\rangle_{00,\gamma}}
{||g||_{H_{00}^{1/2}(\gamma)}},$$
where $\langle\cdot,\cdot\rangle_{00,\gamma}$ denotes the duality paring between $H^{-1/2}_{00}(\gamma)$ and
$H^{1/2}_{00}(\gamma)$.
\section{Problem formulation and mortar finite element method}
Consider a second order parabolic initial-boundary value problem:
\begin{align}
\dot{u}(x,t)-\nabla\cdot(\alpha(x)\nabla &u(x,t)) =f(x,t)\hspace{0.3cm}\text{in}~\Omega\times (0,T],\label{h1}\\
&u(x,t)=0\hspace{0.3cm}\text{on}~\partial\Omega\times (0,T],\label{h2}\\
&u(x,0)=u_{0}(x)\hspace{0.3cm}\text{in}~\Omega,\label{h4}
\end{align}
where $T$ is the fixed final time, $\dot{u}=\frac{\partial u}{\partial t}$, $\nabla\equiv(\frac{\partial}
{\partial x_{1}},\frac{\partial}{\partial x_{2}})$,
$f$ and $u_{0}$ are appropriate smooth functions. Assume that $\alpha(x)$ is smooth and satisfies $0<m_{l}\leq
\alpha(x)\leq m_{u}$, for some positive constants $m_{l}$ and $m_{u}$ and for all $x\in\overline{\Omega}$.

The weak formulation of problem \eqref{h1}-\eqref{h4} is to find $u:[0,T]\rightarrow H^{1}_{0}(\Omega)$
such that
\begin{align}
(\dot{u},v)+&a(u,v)=f(v)\hspace{0.2cm}\forall v\in H^{1}_{0}(\Omega),\label{h43}\\
&u(0)=u_{0},\label{h44}
\end{align}
where
$$a(u,v)=\sum^{n_{0}}_{i=1}\int_{\Omega_{i}}\alpha\nabla u\cdot\nabla v~dx\hspace{0.2cm}\text{and}\hspace{0.2cm}f(v)
=\int_{\Omega}fv~dx.$$

We note that, \eqref{h43}-\eqref{h44} has a unique solution (cf. \cite{evans}) and the standard finite element methods with above formulation are extensively studied, cf. \cite{thomee}.

Let $\Omega$ be partitioned into non-overlapping polygonal subdomains $\{\Omega_{i}\}^{n_{0}}_{1}$  and
\begin{equation*}
\overline{\Omega}=\bigcup^{n_{0}}_{i=1}\overline{\Omega}_{i}.
\end{equation*}
This partition is said to be geometrically conforming if $\partial\Omega_{i}\cap\partial\Omega_{j}$
$(i\neq j)$ is  a vertex  or a whole edge of both subdomains $\Omega_{i}$ and $\Omega_{j}$ or is empty.
Otherwise it is called  geometrically nonconforming. Here, we discuss both the cases.

Now we define
\begin{equation}
X=\{v\in L^{2}(\Omega):v_{i}=v|_{\Omega_{i}}\in H^{1}_{D}(\Omega_{i}),1\leq i\leq n_{0}\},
\end{equation}
with the norm
\begin{equation}
{||v||}^{2}_{X}=\sum^{n_{0}}_{i=1}{||v||}^{2}_{H^{1}(\Omega_{i})},
\end{equation}
where 
\begin{equation}
H^{1}_{D}(\Omega_{i})=\{v\in H^{1}(\Omega_{i}):v|_{\partial\Omega\cap\partial\Omega_{i}}=0\}.
\end{equation}
Let the interface $\Gamma$ be defined as the union of the interfaces $\Gamma_{i,j}$
$(=\partial\Omega_{i}\cap\partial\Omega_{j})$ i.e., $\Gamma=\bigcup_{i,j}\Gamma_{i,j}$,
and be further partitioned into a set of disjoint line segments $\gamma_{j},j=1,2,\cdots,L$.
We denote $Z=\{\gamma_{1},\cdots,\gamma_{L}\}$. Also denote $\mathcal{A}$ to be the set of all
vertices of $\Omega_{i}$ for $1\leq i\leq n_{0}$.

Let $\mathcal{T}_{h_{i}}$ be a family of triangulation (in the sense of Ciarlet \cite{ciar}) of $\Omega_{i}$,
with triangles and parallelograms $K$ having the diameter $h_{K}$. Here $h_i$ is the mesh parameter defined as $h_{i}=\displaystyle\max_{K\in\mathcal{T}_{h_{i}}}{h_{K}}$. We assume the  following quasiuniformity and shape regularity conditions:
for each $K\in\mathcal{T}_{h_{i}}$  there exist positive constants  $\kappa$ and $\varrho$
independent of $h_{i}$ such that
\begin{equation}\label{h83}
\frac{h_{i}}{h_{K}}\leq\kappa
\end{equation}
and 
\begin{equation}\label{h84}
\frac{h_{K}}{\sigma_{K}}\leq\varrho,
\end{equation}
where   $\sigma_{K}=\sup\{\text{diam}(B):B ~\text{a ball in}~ K\}$.

We set $\mathcal{P}^{1}_{k}(K)$ to be the space of all polynomials having degree $\leq k$ and $\mathcal{P}^{2}_{k}(K)$
be the space of all polynomials having degree $\leq k$ in each variables. Also assume  $\mathcal{P}_{k}(K)$ denotes
$\mathcal{P}^{1}_{k}(K)$ if $K$ is a triangle and $\mathcal{P}^{2}_{k}(K)$ if $K$ is a square. For simplicity we assume
$k$ to be uniform in each subdomains although we can take polynomials of different degree in different subdomains as well
as in different elements.

We define finite dimensional subspace on each subdomain $\Omega_{i}$ as
\begin{equation}
X_{h_{i},k}=\{v\in H^{1}(\Omega_{i}):v|_{K}\in\mathcal{P}_{k}(K)\hspace{0.2cm}for\hspace{0.2cm}K
\in\mathcal{T}_{h_{i}},v=0\hspace{0.2cm}on\hspace{0.2cm}\partial\Omega_{i}\cap\partial\Omega\}.
\end{equation}
Here, $X_{h_{i},k}$ are conforming spaces over $\Omega_{i}$ i.e., they contain continuous functions in
$H^{1}(\Omega_{i})$ that vanishes on $\partial\Omega$.
Now we define the global space $X_{h,k}\subset X$ as
\begin{equation}
X_{h,k}=\{v\in L^{2}(\Omega):v_{i}\in X_{h_{i},k}\}.
\end{equation}
Note that, functions in $X_{h,k}$ do not satisfy any continuity condition across the interface.
Let $\gamma\in Z$ such that $\gamma\subset\Gamma_{i,j}$. We define two types of index associated with
$\gamma$, $i=M(\gamma)$  to be mortar index and $j=NM(\gamma)$ to be nonmortar index. Since independent
discretization is possible in different subdomains, we assume two separate meshes $\mathcal{T}^{h}_{M(\gamma)}$
and $\mathcal{T}^{h}_{NM(\gamma)}$, defined on $\gamma$, being inherited from $\Omega_{i}$ and $\Omega_{j}$), respectively.
We define $W^{M}(\gamma)$ to be  the  mortar trace space by
\begin{equation}
W^{M}(\gamma)=\{v_i|_{\gamma}:v_i\in X_{h_{i},k}\}.
\end{equation}
Similarly, we can define $W^{NM}(\gamma)$ for nonmortar trace.
For given $v\in X_{h,k}$, we denote the mortar and nonmortar traces of $v$ on $\gamma$ by $v^{M}_{\gamma}$
and $v^{NM}_{\gamma}$ respectively. In order
to impose the weak continuity across the common interface, we  confine the space $X_{h,k}$ by making the
jump $[\![v]\!]=v^{M}_{\gamma}-v^{NM}_{\gamma}$ orthogonal to a suitable multiplier space. This is called gluing
technique or mortaring technique and is accomplished with the help of  multiplier spaces
$S^{NM}_{h,k}(\gamma)$ defined on the nonmortar trace mesh $\mathcal{T}^{h}_{NM(\gamma)}$.

Let the subintervals of the cited mesh on $\gamma$ be given by $I_{i}=[x_{i},x_{i+1}]$, $0\leq i\leq l$. We define
\begin{align}
S^{NM}_{h,k}(\gamma)=\{\chi\in \mathcal{C}(\gamma):\chi|_{I_{i}}\in\mathcal{P}_{k}(I_{i}),&i=1,\cdots,l-1,\nonumber\\
&\chi|_{I_{j}}
\in\mathcal{P}_{k-1}(I_{j}),j=0,l\}.
\end{align}
Now we define the constrained space $V_{h,k}\subset X_{h,k}$ as the following,
\begin{equation}\label{h8}
V_{h,k}=\bigg\{v\in X_{h,k}:\int_{\gamma}(v^{M}_{\gamma}-v^{NM}_{\gamma})\chi~d\tau=0~
\forall\chi\in S^{NM}_{h,k}(\gamma)~\forall\gamma\in Z\bigg\}.
\end{equation}
Mortar formulation of the problem \eqref{h1}-\eqref{h4} is to find $u_{h,k}:[0,T]\rightarrow V_{h,k}$ such that
\begin{align}
(\dot{u}_{h,k},v_{h,k})+a(u_{h,k},v_{h,k})&=f(v_{h,k})\hspace{0.2cm}\forall v_{h,k}\in V_{h,k},\label{h34}\\
u_{h,k}(0)&=u_{0,h,k},\label{ini}
\end{align}
where $u_{0,h,k}$ is a suitable approximation of $u_{0}$ in $V_{h,k}$.
We note that, from the properties of the coefficient $\alpha(x)$, it is evident that, the bilinear form $a(\cdot,\cdot)$
satisfy boundedness  property: for $v,w\in X$, there exist a constant $M>0$   such that
\begin{equation}\label{h70}
a(v,w)\leq M{||v||}_{X}{||w||}_{X}.
\end{equation}
Also, $a(\cdot,\cdot)$ satisfies the coercivity property (cf. \cite{bernardy,sesh2}) for the functions in $V_{h,k}$:
for $v_{h,k}\in V_{h,k}$, there exists a constant $c>0$ independent of $h$ such that
\begin{equation}\label{h71}
a(v_{h,k},v_{h,k})\geq c{||v_{h,k}||}^{2}_{X}.
\end{equation}
We note that \eqref{h34} is equivalent to a linear system of ordinary differential equations and the corresponding matrix is positive definite. Therefore the existence and uniqueness of the solution to \eqref{h34} on $[0,T]$ follows from the Picard's theorem.

We let $C$ denote a generic positive constant throughout the paper. 

\section{Approximation properties}
For all $\gamma\in Z$, let $W^{NM}_{0}(\gamma)\subset W^{NM}(\gamma)$ denote the space of functions vanishing
at the end points of $\gamma$. Consider the operator
$\Pi^{h,k}_{\gamma}:L^{2}(\gamma)\rightarrow W^{NM}_{0}(\gamma)$ defined as the following, that is, for $v\in
L^{2}(\gamma)$ and $\gamma\in Z$, $\Pi^{h,k}_{\gamma}v\in W^{NM}_{0}(\gamma)$ satisfies
\begin{equation}\label{h72}
\int_{\gamma}(\Pi^{h,k}_{\gamma}v)\chi~ds=\int_{\gamma}v\chi~ds\hspace{0.2cm}\forall\chi\in S^{NM}_{h,k}(\gamma).
\end{equation}
The proof of the following lemma  can be found in \cite{sesh}.
\begin{lemma}
\label{lem3}
For any $\nu\geq 0$ and $\gamma\in Z$ the following estimate holds for all $\varphi\in H^{1/2+\nu}(\gamma)$:
\begin{equation}\label{h73}
{||\varphi-\Pi^{h,k}_{\gamma}\varphi||}_{H^{1/2}_{00}(\gamma)}\leq C~h^{\eta}~k^{-\nu}
(\log k)^{1/2}{||\varphi||}_{H^{1/2+\nu}(\gamma)},
\end{equation}
where $\eta=\min(\nu,k)$ and $C$ a positive constant independent of $h$ and $k$.
\end{lemma}
Now we recall the following two lemmas, the proofs of which can be found in \cite{babuska}.
\begin{lemma}
\label{extn}
For each $\gamma\in Z$ and index $i$, such that $\gamma\subset\partial\Omega_{i}$ and $i=NM(\gamma),$ there exists an extension operator $R^{\gamma}_{h_i,k}:W^{NM}_{0}(\gamma)\rightarrow X_{h_i,k}$ satisfying, for all $z\in W^{NM}_{0}(\gamma)$,
$$R^{\gamma}_{h_i,k} z=z\hspace{0.2cm}\text{on}\hspace{0.2cm}\gamma,\hspace{0.2cm}
R^{\gamma}_{h_i,k}z=0\hspace{0.2cm}\text{on}\hspace{0.2cm}\partial\Omega_{i}\backslash\gamma,$$
\begin{equation}\label{h6}
{||R^{\gamma}_{h_i,k}z||}_{H^1(\Omega_{i})}\leq C~{||z||}_{H^{1/2}_{00}(\gamma)},
\end{equation}
where $C$ is a positive constant independent of $h, k, z$. 
\end{lemma}
\begin{lemma}
\label{lem1}
Let $K$ be a triangle or parallelogram with vertices $\{\mathcal{A}_{i}\}$ satisfying \eqref{h83} and \eqref{h84}.
Also assume $v\in H^{\nu}(K)$. Then there exists a positive constant $C$ independent of $v$, $k$ and $h_{K}$, but depends upon $\nu$, $\kappa$ and $\varrho$, and a sequence $\mathcal{I}^{h_{K}}_{k}v\in\mathcal{P}_{k}(K)$, such that for
any $0\leq \nu_{1}\leq \nu$
\begin{equation}\label{s30}
{||v-\mathcal{I}^{h_{K}}_{k}v||}_{H^{\nu_{1}}(K)}\leq C~h_{K}^{\eta-\nu_{1}}k^{-(\nu-\nu_{1})}{||v||}_{H^{\nu}(K)},
\end{equation}
where $\eta=\min(\nu,k+1)$.
If $\nu>3/2$ then we can assume that $\mathcal{I}^{h_{K}}_{k}v(\mathcal{A}_{i})=v(\mathcal{A}_{i})$. Further, for $\sigma\in[0,1]$
\begin{equation}\label{s31}
{||v-\mathcal{I}^{h_{K}}_{k}v||}_{\sigma,\gamma_{K}}\leq C~h_{K}^{\eta-1/2-\sigma}k^{-(\nu-1/2-\sigma)}{||v||}_{H^{\nu}(K)},
\end{equation}
where $\gamma_{K}$ is a side of $K$, while ${||\cdot||}_{\sigma,\gamma_{K}}$ is the norm defined on the interpolation space 
$[L^{2}(\gamma_{K}),H^{1}_{0}(\gamma_{K})]_{\sigma}$ and ${||\cdot||}_{0,\gamma_{K}}$ the norm defined on $L^{2}(\gamma_{K})$.
\end{lemma}

For $v\in X$, we choose $I^{h}_{k}v\in X_{h,k}$ such that $I^{h}_{k}v$ equals $\mathcal{I}^{h_{K}}_{k}v_{i}$
on each $K\in\mathcal{T}_{h_{i}},1\leq i\leq n_{0}$ and define the operator $Q_{h,k}$ as
\begin{equation}\label{h79}
Q_{h,k}v=I^{h}_{k}v+\sum_{\gamma\in Z}w_{\gamma},
\end{equation}
where $w_{\gamma}=0$ when $\gamma$ is a mortar segment otherwise  $w_{\gamma}=R^{\gamma}_{h_i,k}(\Pi^{h,k}_{\gamma}(I^{h}_{k}v^{M}_{\gamma}-I^{h}_{k}v^{NM}_{\gamma}))$,
$i=NM(\gamma)$. Clearly $Q_{h,k}v$ belongs to $V_{h,k}$, and the following result holds.
\begin{lemma}
\label{lem4}
Let $v\in H^{1}_{0}(\Omega)$ such that $v_{i}\in H^{\nu}(\Omega_{i})$, $\nu> 3/2$. Then there exists a positive constant $C$ independent of $h$ and
$k$ such that
\begin{equation}\label{h81}
{||v-Q_{h,k}v||}_{X}\leq C~h^{\eta-1}~k^{-(\nu-1)}(\log k)^{1/2}\sum^{n_{0}}_{i=1}{||v||}_{H^{\nu}(\Omega_{i})},
\end{equation}
where $\eta=\min(\nu,k+1)$.
\end{lemma}
\begin{proof}
Since for $i=NM(\gamma)$, $w_{\gamma}\in X_{h_i,k}$, from lemmas \ref{lem3}, \ref{extn}, \ref{lem1}
and triangle inequality we find
\begin{align}
{||w_{\gamma}||}_{H^{1}(\Omega_{i})}
&\leq {||(I^{h}_{k}v^{M}_{\gamma}-I^{h}_{k}v^{NM}_{\gamma})-\Pi^{h,k}_{\gamma}
(I^{h}_{k}v^{M}_{\gamma}-I^{h}_{k}v^{NM}_{\gamma})||}_{H^{1/2}_{00}(\gamma)}\nonumber\\
&\hspace{0.5cm}+{||I^{h}_{k}v^{M}_{\gamma}-I^{h}_{k}v^{NM}_{\gamma}||}_{H^{1/2}_{00}(\gamma)}\nonumber\\
&\leq C~h^{\eta-1}~k^{-(\nu-1)}(\log k)^{1/2}\sum^{n_{0}}_{i=1}{||v||}_{H^{\nu}(\Omega_{i})}.\nonumber
\end{align}
Hence \eqref{h81} follows from  \eqref{h79} and  Lemma \ref{lem1}.
\end{proof}
The lemma below is used to compute the consistency error of the approximation. 
\begin{lemma}
\label{nonconf}
Assume that for $t\in[0,T]$, $u(t)\in H^{1}_{0}(\Omega)$ and $u_{i}(t),\dot{u}_{i}(t)\in H^{\nu}(\Omega_{i})$, $\nu> 3/2$. Then for a geometrically nonconforming partition of $\Omega$, the following estimates hold
for any $w_{h,k}\in V_{h,k}$:
\begin{equation}\label{h76}
\frac{\displaystyle\int_{\Gamma}\alpha\nabla u\cdot n[\![w_{h,k}]\!]d\tau}{{||w_{h,k}||}_{X}}
\leq C ~h^{\eta-1}~k^{-(\nu-1)}\bigg|\log\frac{k}{h}\bigg|^{1/2}\sum^{n_{0}}_{i=1}{||u||}_{H^{\nu}(\Omega_{i})}
\end{equation}
and
\begin{equation}\label{h78}
\frac{\displaystyle\int_{\Gamma}\alpha\nabla \dot{u}\cdot n[\![w_{h,k}]\!]d\tau}{{||w_{h,k}||}_{X}}
\leq C ~h^{\eta-1}~k^{-(\nu-1)}\bigg|\log\frac{k}{h}\bigg|^{1/2}\sum^{n_{0}}_{i=1}{||\dot{u}||}_{H^{\nu}(\Omega_{i})},
\end{equation}
where $\eta=\min(\nu,k+1)$, and $C$ is a positive constant independent of $h$ and $k$.
\end{lemma}
\begin{proof}
Since the partition of $\Omega$ is geometrically nonconforming, the jump $[\![w_{h,k}]\!]$ belong to
$H^{1/2-\epsilon}(\gamma)$, $0<\epsilon\leq 1/2$. On each $\gamma\in Z$, from the definition of
$V_{h,k}$ and $\Pi^{h,k}_{\gamma}$, we can write
\begin{align}
\int_{\gamma}\alpha\nabla u\cdot n[\![w_{h,k}]\!]d\tau
&=\int_{\gamma}\big(\alpha\nabla u\cdot n-\psi\big)[\![w_{h,k}]\!]d\tau,\hspace{0.2cm}\forall
\psi\in S^{NM}_{h,k}(\gamma)\nonumber\\
&\leq\inf_{\psi\in S^{NM}_{h,k}(\gamma)}{||\alpha\nabla u\cdot n-\psi||}_{H^{-1/2+\epsilon}(\gamma)}
{||[\![w_{h,k}]\!]||}_{H^{1/2-\epsilon}(\gamma)}.\label{h77}
\end{align}
Now the best approximation property gives
\begin{equation}
\inf_{\psi\in S^{NM}_{h,k}(\gamma)}{||\alpha\nabla u\cdot n-\psi||}_{H^{-1/2+\epsilon}(\gamma)}
\leq C~h^{\eta-1-\epsilon}k^{-(\nu-1-\epsilon)}\sum^{n_{0}}_{i=1}{||u||}_{H^{\nu}(\Omega_{i})},
\end{equation}
and as in \cite{bernardy1} we observe that
\begin{equation}
{||[\![w_{h,k}]\!]||}_{H^{1/2-\epsilon}(\gamma)}
\leq C~\epsilon^{-1/2}{||w_{h,k}||}_{X}.
\end{equation}
From \eqref{h77}, we obtain
\begin{align}
\int_{\gamma}\alpha\nabla u\cdot n[\![w_{h,k}]\!]d\tau
&\leq C~h^{\eta-1-\epsilon}k^{-(\nu-1-\epsilon)}\epsilon^{-1/2}\sum^{n_{0}}_{i=1}
{||u||}_{H^{\nu}(\Omega_{i})}{||w_{h,k}||}_{X}.\nonumber
\end{align}
Taking $\epsilon=\big(\log \frac{k}{h}\big)^{-1}$, we see that $h^{-\epsilon}k^{\epsilon}$ is a constant
($\approx$ 2.7183 if $h<1$).  Summing over all $\gamma\in Z$, we find
\begin{align}
\int_{\Gamma}\alpha\nabla u\cdot n[\![w_{h,k}]\!]d\tau
\leq C~h^{\eta-1}k^{-(\nu-1)}\bigg|\log\frac{k}{h}\bigg|^{1/2}\sum^{n_{0}}_{i=1}
{||u||}_{H^{\nu}(\Omega_{i})}{||w_{h,k}||}_{X}.\nonumber
\end{align}
Hence, \eqref{h76} follows. Replacing $u$ by $\dot{u}$ and proceeding in the same way as above,
\eqref{h78} follows.
\end{proof}
We define a modified elliptic projection $P_{h,k}$ from $\displaystyle{\prod^{n_{0}}_{i=1} H^{3/2}(\Omega_{i})}$ onto $V_{h,k}$ (cf. \cite{patel}) as follows, i.e., for a given $u\in {\displaystyle\prod^{n_{0}}_{i=1} H^{3/2}(\Omega_{i})}$, find $P_{h,k}u\in V_{h,k}$ such that
\begin{equation}\label{h7}
a(u-P_{h,k}u,\chi)=\sum_{\gamma\in Z}\int_{\gamma}\alpha\nabla u\cdot n[\![ \chi ]\!]~d\tau\hspace{0.2cm}\forall\chi\in V_{h,k}.
\end{equation}
Note that, since the bilinear form $a(\cdot,\cdot)$ satisfies the coercivity property \eqref{h71},
for a given $u\in X$, the problem \eqref{h7} has a unique solution $P_{h,k}u\in V_{h,k}$.  The following lemma is on the estimate of the above intermediate finite element approximation $P_{h,k}u\in V_{h,k}$ of the exact solution $u$ of \eqref{h13}-\eqref{h14}. 
\begin{lemma}
\label{proj}
Assume that for $t\in[0,T]$, $u(t)\in H^{1}_{0}(\Omega)$ and $u_{i}(t),\dot{u}_{i}(t)\in H^{\nu}(\Omega_{i})$, $\nu> 3/2$.
Then for a geometrically nonconforming partition of $\Omega$, there exists a positive constant $C$ independent of $h$, $k$ such that
\begin{align}\label{h29}
(\log k)^{-1/2}&{||u-P_{h,k}u||}_{L^{2}(\Omega)}+h~k^{-1}{||u-P_{h,k}u||}_{X}\nonumber\\
&\leq C~h^{\eta}~k^{-\nu}\bigg|\log\frac{k}{h}\bigg|^{1/2}(\log k)^{1/2}
\sum^{n_{0}}_{i=1}{||u||}_{H^{\nu}(\Omega_{i})}
\end{align}
and
\begin{align}\label{h30}
(\log k)^{-1/2}&{||\dot{u}-P_{h,k}\dot{u}||}_{L^{2}(\Omega)}+h~k^{-1}{||\dot{u}-P_{h,k}\dot{u}||}_{X}\nonumber\\
&\leq
C~h^{\eta}~k^{-\nu}\bigg|\log\frac{k}{h}\bigg|^{1/2}(\log k)^{1/2}\sum^{n_{0}}_{i=1}{||\dot{u}||}_{H^{\nu}(\Omega_{i})},
\end{align}
where $\eta=\min(\nu,k+1)$.
\end{lemma}
\begin{proof}
Using lemmas \ref{lem1}, \ref{lem4} and \ref{nonconf}, and proceeding as in the proof of the Lemma 3.5 of \cite{patel}, the \eqref{h29} and \eqref{h30} follows. 
\end{proof}
\begin{remark}
We may rewrite  \eqref{h29} 
for a geometrically conforming partition of $\Omega$ as
\begin{align*}
(\log k)^{-1/2}&{||u-P_{h,k}u||}_{L^{2}(\Omega)}+h~k^{-1}{||u-P_{h,k}u||}_{X}\\
&\leq C~h^{\eta}~k^{-\nu}(\log k)^{1/2}
\sum^{n_{0}}_{i=1}{||u||}_{H^{\nu}(\Omega_{i})}.
\end{align*}
 Similarly, \eqref{h30} can be written for the geometrically conforming case. 
\end{remark}
\section{Error estimates for the semidiscrete method}
For $v\in X$ and for $t\in(0,T]$, from equation \eqref{h1}-\eqref{h4}, we find
\begin{align}
(\dot{u},v)+a(u,v)&=f(v)+\sum_{\gamma\in Z}\int_{\gamma}\alpha\nabla u\cdot n[\![v]\!]d\tau,\label{h13}\\
u(0)&=u_{0}.\label{h14}
\end{align}
\begin{theorem}
\label{thm1}
Let $u$ and $u_{h,k}$ be the solutions of \eqref{h13}-\eqref{h14} and \eqref{h34}-\eqref{ini}, respectively.
Further, let $u_{0,h,k}=I^{h}_{k}u_{0}$ or $P_{h,k}u_{0}$. Then for  a geometrically nonconforming partition of $\Omega$ there exists a positive constant $C$ independent of $h$, $k$ and
$u$ such that for $t\in(0,T]$,
\begin{equation}\label{h19}
{||(u-u_{h,k})(t)||}_{L^{2}(\Omega)}\leq Ch^{\eta}~k^{-\nu}\log k\bigg|\log\frac{k}{h}\bigg|^{\frac{1}{2}}\sum^{n_{0}}_{i=1}\bigg({||u_{0}||}_{H^{\nu}(\Omega_{i})}+{||\dot{u}||}_{L^{2}(0,T;H^{\nu}(\Omega_{i}))}\bigg)
\end{equation}
and
\begin{align}\label{h20}
{||(u-u_{h,k})(t)||}_{X}\leq Ch^{\eta-1}~k^{-(\nu-1)}(\log k)^{\frac{1}{2}}\bigg|\log\frac{k}{h}\bigg|^{\frac{1}{2}}\sum^{n_{0}}_{i=1}\bigg(&{||u_{0}||}_{H^{\nu}(\Omega_{i})}\nonumber\\
&+{||\dot{u}||}_{L^{2}(0,T;H^{\nu}(\Omega_{i}))}\bigg)
\end{align}
where $\eta=\min(\nu,k+1)$.
\end{theorem}
\begin{proof}
Writing $u-u_{h,k}=\underbrace{u-P_{h,k}u}+\underbrace{P_{h,k}u-u_{h,k}}=\rho+\theta.$ From lemma
\ref{proj}, estimates for $\rho$ are known. So it is enough to estimate $\theta$. From \eqref{h34},
\eqref{h7} and \eqref{h13}, we obtain
\begin{equation}\label{h15}
(\dot{\theta},\chi)+a(\theta,\chi)=-(\dot{\rho},\chi),\hspace{0.2cm}\forall\chi\in V_{h,k}.
\end{equation}
Replacing $\chi$ with $\theta$ in \eqref{h15}, applying coercivity \eqref{h71} of $a(\cdot,\cdot)$ and
using Young's inequality $ab\leq\frac{\epsilon}{2}a^{2}+\frac{1}{2\epsilon}b^{2},a,b,\epsilon>0$, we arrive at
\begin{align}
\frac{1}{2}\frac{d}{dt}{||\theta||}^{2}_{L^{2}(\Omega)}+
c{||\theta||}^{2}_{X}&\leq{||\dot{\rho}||}_{L^{2}(\Omega)}{||\theta||}_{L^{2}(\Omega)}\nonumber\\
&\leq {||\dot{\rho}||}_{L^{2}(\Omega)}{||\theta||}_{X}\nonumber\\
&\leq\frac{1}{2c}{||\dot{\rho}||}^{2}_{L^{2}(\Omega)}+\frac{c}{2}{||\theta||}^{2}_{X}.\nonumber
\end{align}
Hence $$\frac{d}{dt}{||\theta||}^{2}_{L^{2}(\Omega)}+c{||\theta||}^{2}_{X}\leq\frac{1}{c}
{||\dot{\rho}||}^{2}_{L^{2}(\Omega)}.$$
Integrating from $0$ to $t$, we find
\begin{equation}\label{h16}
{||\theta(t)||}^{2}_{L^{2}(\Omega)}\leq{||\theta(0)||}^{2}_{L^{2}(\Omega)}+\frac{1}{c}\int^{t}_{0}{||\dot{\rho}||}^{2}_{L^{2}(\Omega)}ds. 
\end{equation}
If $u_{0,h,k}=P_{h,k}u_{0}$, then $\theta(0)=0$, otherwise with $u_{0,h,k}=I^{h}_{k}u_{0},$
\begin{align}
{||\theta(0)||}_{L^{2}(\Omega)}
&={||P_{h,k}u_{0}-u_{0,h,k}||}_{L^{2}(\Omega)}\nonumber\\
&\leq {||u_{0}-I^{h}_{k}u_{0}||}_{L^{2}(\Omega)}+{||P_{h,k}u_{0}-u_{0}||}_{L^{2}(\Omega)}\nonumber\\
&\leq C~h^{\eta}~k^{-\nu}\log k\bigg|\log\frac{k}{h}\bigg|^{1/2}\sum^{n_{0}}_{i=1}{||u_{0}||}_{H^{\nu}(\Omega_{i})}\label{h17}.
\end{align}
For the second term on the right hand side of \eqref{h16}, we apply Lemma \ref{proj} to obtain
\begin{equation}\label{h18}
{||\dot{\rho}||}_{L^{2}(\Omega)}={||\dot{u}-P_{h,k}\dot{u}||}_{L^{2}(\Omega)}\leq C~h^{\eta}~k^{-\nu}\log k\bigg|\log\frac{k}{h}\bigg|^{1/2}
\sum^{n_{0}}_{i=1}{||\dot u||}_{H^{\nu}(\Omega_{i})}.
\end{equation}
Substituting \eqref{h17} and \eqref{h18} in \eqref{h16}, we find that
\begin{align}\label{hhh1}
&{||\theta(t)||}^{2}_{L^{2}(\Omega)}\nonumber\\
&\hspace{0.5cm}\leq C~h^{2\eta}k^{-2\nu}(\log k)^{2}\bigg|\log\frac{k}{h}\bigg|\sum^{n_{0}}_{i=1}
\bigg({||u_{0}||}^{2}_{H^{\nu}(\Omega_{i})}
+\int^{t}_{0}{||\dot{u}||}^{2}_{H^{\nu}(\Omega_{i})}ds\bigg).
\end{align}
Since for a function $\varphi$,
\begin{equation}\label{h45}
\varphi(t)=\varphi(0)+\int^{t}_{0}\dot{\varphi}(s)ds,
\end{equation}
we have 
\begin{equation}
{||\rho(t)||}_{L^{2}(\Omega)}\leq C\left({||\rho(0)||}_{L^{2}(\Omega)}+\int^{t}_{0}{||\dot{\rho}(s)||}_{L^{2}(\Omega)}ds\right).\label{hh2}
\end{equation}
Using \eqref{h18}, \eqref{hhh1}, \eqref{hh2},   Lemma \ref{proj} and triangle inequality, \eqref{h19} follows.
For a bound in $X$-norm, substitute $\chi=\dot{\theta}$ in \eqref{h15} and applying Cauchy-Schwarz inequality, we obtain
\begin{equation}
{||\dot{\theta}||}_{L^{2}(\Omega)}^{2}+\frac{1}{2}\frac{d}{dt}
a(\theta,\theta)\leq{||\dot{\rho}||}_{L^{2}(\Omega)}{||\dot{\theta}||}_{L^{2}(\Omega)}
\leq\frac{1}{2}{||\dot{\rho}||}^{2}_{L^{2}(\Omega)}+\frac{1}{2}{||\dot{\theta}||}^{2}_{L^{2}(\Omega)}
\end{equation}
and hence
\begin{equation}\label{h21}
{||\dot{\theta}||}^{2}_{L^{2}(\Omega)}+\frac{d}{dt}a(\theta,\theta)\leq{||\dot{\rho}||}^{2}_{L^{2}(\Omega)}.
\end{equation}
Integrating both side of \eqref{h21} from $0$ to $t$, using boundedness \eqref{h70} and coercivity
 \eqref{h71}  of $a(\cdot,\cdot)$,   we arrive at
\begin{equation}\label{h89}
{||\theta(t)||}^{2}_{X}\leq C(c)\bigg({||\theta(0)||}^{2}_{X}+\int^{t}_{0}{||\dot{\rho}||}^{2}_{L^{2}
(\Omega)}ds\bigg).
\end{equation}
Similarly,  with $u_{0,h,k}=I^{h}_{k}u_{0}$
\begin{align}
{||\theta(0)||}_{X}&={||u_{0}-I^{h}_{k}u_{0}||}_{X}+{||P_{h,k}u_{0}-u_{0}||}_{X}\\
&\leq C~h^{\eta-1}~k^{-(\nu-1)}(\log k)^{1/2}\bigg|\log\frac{k}{h}\bigg|^{1/2}\sum^{n_{0}}_{i=1}{||u_{0}||}_{H^{\nu}(\Omega_{i})}.\label{h88}
\end{align}
Finally, using  \eqref{h88} and  Lemma \ref{proj}, from \eqref{h89}, we find
\begin{align}\label{h35}
{||\theta(t)||}^{2}_{X}\leq C\sum^{n_{0}}_{i=1}\bigg(&h^{2(\eta-1)}~k^{-2(\nu-1)}\log k\bigg|\log\frac{k}{h}\bigg|^{1/2}{||u_{0}||}^{2}_{H^{\nu}(\Omega_{i})}\nonumber\\
&+h^{2\eta}k^{-2\nu}(\log k)^{2}\bigg|\log\frac{k}{h}\bigg|\int^{t}_{0}{||\dot{u}||}^{2}_{H^{\nu}(\Omega_{i})}ds\bigg).
\end{align}
Hence, \eqref{h20} follows.
\end{proof}
\begin{remark}
Taking $u_{0,h,k}=P_{h,k}u_{0}$, we have $\theta(0)=0$, and from \eqref{h35}, we have the following superconvergence
property of $\theta$:
\begin{equation*}
{||\theta(t)||}^{2}_{X}\leq C~h^{2\eta}~k^{-2\nu}(\log k)^{2}\bigg|\log\frac{k}{h}\bigg|\sum^{n_{0}}_{i=1}\int^{t}_{0}{||\dot{u}||}^{2}_{H^{\nu}(\Omega_{i})}ds.
\end{equation*}
\end{remark}
\begin{remark}
Note that for a geometrically conforming partition of the domain, \eqref{h20} becomes
\begin{equation*}
{||(u-u_{h,k})(t)||}_{X}\leq C~h^{\eta-1}~k^{-(\nu-1)}(\log k)^{1/2}\sum^{n_{0}}_{i=1}\big({||u_{0}||}_{H^{\nu}(\Omega_{i})}+{||\dot{u}||}_{L^{2}(0,T;H^{\nu}(\Omega_{i}))}\big).
\end{equation*}
\end{remark}
\begin{remark}
A right combination of nonquasiuniform meshes and the polynomial degree  may lead to exponential decay in the spatial error even in the presence of singularity in the domain (cf. \cite{suri}). 
\end{remark}
\section{Superconvergence estimates in negative norms}
\label{ch33}
In the next lemma we derive a negative norm estimate for $\nu>2$ which is analogous to theorem 5.1 of \cite{thomee}.
\begin{lemma}
\label{negp}
Assume $u$ satisfies the hypothesis of Lemma \ref{proj}. Then the following negative estimate holds:
\begin{equation*}
{||u-P_{hk}u||}_{H^{-s}(\Omega)}\leq C~h^{\eta+s}k^{-(q+s)}\log k\bigg|\log\frac{k}{h}\bigg|^{1/2}{||u||}_{H^{q}(\Omega)},
\end{equation*}
for $0\leq s\leq \nu-2,~3/2< q\leq \nu$, where $\eta=\min\{q,k+1\}$.
\end{lemma}
\begin{proof}
Consider for $i=1,\cdots,n_0$, $z_{i}\in H^{s+2}(\Omega_i)\cap H^{1}_{D}(\Omega_i)$ such that
\begin{align}
-\nabla\cdot(\alpha_i(x)\nabla z_{i})&=\phi_{i}\hspace{0.2cm}\text{in}\hspace{0.2cm}\Omega_{i},\label{h5}\\
z_{i}&=0\hspace{0.2cm}\text{on}\hspace{0.2cm}\partial\Omega_{i}\cap\partial\Omega,\nonumber\\
[\![z]\!]=0,~[\![\alpha\nabla z\cdot n]\!]&=0\hspace{0.2cm}\text{along}\hspace{0.2cm}\Gamma,\nonumber
\end{align}
with the regularity condition
\begin{equation}\label{regm}
\sum^{n_{0}}_{i=1}{||z||}_{H^{s+2}(\Omega_{i})}\leq C{||\phi||}_{H^{s}(\Omega)}.
\end{equation}
Multiplying  \eqref{h5} with  $u-P_{hk}u$ and using Green's formula, we find 
\begin{align*}
(u-P_{hk}u,\phi)
&=-\sum^{n_{0}}_{i=1}\int_{\Omega_{i}}(u-P_{hk}u)\nabla\cdot(\alpha_i(x)\nabla z_i)dx\\
&=a(u-P_{hk}u,z-Q_{hk}z)+a(u-P_{hk}u,Q_{hk}z)\\
&\hspace{0.5cm}-\sum_{\gamma\in Z}\int_{\gamma}(\alpha\nabla z\cdot n-\psi){[\![u-P_{hk}u]\!]}d\tau,~\psi\in S^{NM}_{h,k}(\gamma).
\end{align*}
Proceeding as in Lemma \ref{proj} and using the condition \eqref{regm}, the lemma follows.
\end{proof}
Let the solution operator of a self-adjoint elliptic problem
\begin{align}
-\nabla\cdot(\alpha(x)\nabla u(x))&=f(x)\hspace{0.3cm}\text{in}~~~\Omega,\label{ell1}\\
u(x)&=0\hspace{0.3cm}\text{on}~~~\partial\Omega,\label{ell2}
\end{align}
be $A:L^{2}(\Omega)\rightarrow H^{1}_{0}(\Omega)$ such that $u=Af$ is the  solution of \eqref{ell1}-\eqref{ell2}. We note that, $A$ is a 
self-adjoint operator with respect to $(\cdot,\cdot)$ by the self-adjoint property of the bilinear form  $a(\cdot,\cdot)$,
that is $$(f,Ag)=a(Af,Ag)=(Af,g).$$
Also $A$ is positive definite on $L^{2}(\Omega)$: For, since  $(f,Af)=a(Af,Af)\geq c{||Af||}^{2}_{H^{1}(\Omega)}\geq 0$, for $c>0$, $(f,Af)=0$ implies  $f=-\nabla\cdot(\alpha(x)\nabla Af)=0$.
Now we can define an alternative negative norm  by
\begin{equation}\label{def1}
{|v|}_{-s}={||A^{s/2}v||}_{L^{2}(\Omega)}=(A^{s}v,v)^{1/2},\hspace{0.2cm}\text{for}\hspace{0.2cm}s\geq 0,
\end{equation}
which is equivalent  to the negative norm introduced 
earlier and is more convenient to use for the analysis of parabolic problems. 
The following lemma can be found in  \cite[page 71]{thomee}.
\begin{lemma}
\label{eqv}
For a non-negative integer $s$, the norms ${|\cdot|}_{-s}$ and ${||\cdot||}_{H^{-s}(\Omega)}$ are equivalent.
\end{lemma}
Let $A_{hk}:L^{2}(\Omega)\rightarrow V_{h,k}$ such that $u_{h,k}=A_{hk}f$ is the mortar approximate solution to the problem \eqref{ell1}-\eqref{ell2} which is also a self-adjoint operator with respect to $(\cdot,\cdot)$: i.e.,
$$(f,A_{hk}g)=a(A_{hk}f,A_{hk}g)=(A_{hk}f,g)~\forall~f,g\in L^{2}(\Omega).$$
We define a discrete negative semi-norm on $L^2(\Omega)$ as
\begin{equation}\label{def2}
{|v|}_{-s,hk}={||A^{s/2}_{hk}v||}_{L^{2}(\Omega)}=(A^{s}_{hk}v,v)^{1/2},\hspace{0.2cm}\text{for}\hspace{0.2cm}s\geq 0,
\end{equation}
which correspond to discrete semi-inner product	 ${(v,w)}_{-s,hk}=(A^{s}_{hk}v,w)$. Note that $A_{hk}$ is positive semidefinite on $L^{2}(\Omega)$, i.e., from \eqref{h71}, 
\begin{equation}\label{kanha}
(A_{hk}f,f)=a(A_{hk}f,A_{hk}f)\geq c{||A_{hk}f||}^{2}_{X}\geq 0.
\end{equation}
In fact $A_{hk}$ is positive definite on $V_{h,k}$: For, assume $f_{h,k}\in V_{h,k}$ such that $$(A_{hk}f_{h,k},f_{h,k})=0.$$ Then from  \eqref{kanha}, $A_{hk}f_{h,k}=0$  implies ${||f_{h,k}||}^{2}_{L^{2}(\Omega)}=a(A_{hk}f_{h,k},f_{h,k})=0$ and $f_{h,k}=0$.
Hence  ${|v|}_{-s,hk}$ and ${(v,w)}_{-s,hk}$ defines a norm and an inner product on $V_{h,k}$ respectively.

The following lemma is an immediate consequence of  Lemma \ref{negp} and Lemma \ref{eqv}.
\begin{lemma}
\label{neproj2}
For $0\leq s\leq \nu-2,~2\leq q\leq \nu$, there exists a positive constant $C$ independent of $h$ and $k$ such that
\begin{align}
{|Af-A_{hk}f|}_{-s}&={|Af-P_{hk}Af|}_{-s}\nonumber\\
&={|u-P_{hk}u|}_{-s}\nonumber\\
&\leq C~h^{\eta+s}k^{-(q+s)}\log k\bigg|\log\frac{k}{h}\bigg|^{1/2}{||u||}_{H^{q}(\Omega)},\nonumber\\
&\leq C~h^{\eta+s}k^{-(q+s)}\log k\bigg|\log\frac{k}{h}\bigg|^{1/2}{||f||}_{H^{q-2}(\Omega)},\label{negm2}
\end{align}
where $\eta=\min\{q,k+1\}$.
\end{lemma}
Let  $\{\lambda_j\}^{\infty}_{j=1}$ and $\{\phi_j\}^{\infty}_{j=1}$ be the eigenvalues and orthonormal eigenfunctions
of $A^{-1}$ respectively.  We define another norm  on $\dot H^{s}(\Omega)=\{v\in H^s(\Omega):(A^{-1})^j v=0,~j<s/2\}$ which is equivalent to the standard Sobolev norm ${||\phi||}_{H^{s}(\Omega)}$ (see lemma 3.1 of \cite{thomee})  defined as: for
$s\geq 0$
\begin{equation}\label{spec1}
{|\phi|}_{s}=\big((A^{-1})^{s}\phi,\phi\big)^{1/2}=\left(\sum^{\infty}_{j=1}\lambda^{s}_{j}(\psi,\phi_j)^2\right)^{1/2}.
\end{equation}
The eigenvalues and
 eigenfunctions of the compact operator $A$ are  $\{\lambda^{-1}_j\}^{\infty}_{j=1}$ and $\{\phi_j\}^{\infty}_{j=1}$  respectively, and we have
\begin{equation}\label{spec2}
{|v|}_{-s}=(A^{s}v,v)^{1/2}=\left(\sum^{\infty}_{j=1}\lambda^{-s}_{j}(v,\phi_j)^2\right)^{1/2}.
\end{equation}
In the next lemma it is shown that the discrete negative semi-norm is equivalent to the corresponding continuous negative norm, with a mild error.
\begin{lemma}
\label{negeq}
For a non-negative integer $s$ with $0\leq s\leq \nu$, there exist positive constants $C$ independent of $h$ and $k$ such that
\begin{align}
&{|v|}_{-s,hk}\leq C\left({|v|}_{-s}+h^{s}k^s\log k\bigg|\log\frac{k}{h}\bigg|^{1/2}{||v||}_{L^{2}(\Omega)}\right),\label{neg5}\\
&{|v|}_{-s}\leq C\left({|v|}_{-s,hk}+h^{s}k^s\log k\bigg|\log\frac{k}{h}\bigg|^{1/2}{||v||}_{L^{2}(\Omega)}\right).\label{neg6}
\end{align}
\end{lemma}
\begin{proof}
For $s=0$ the results are trivial. For the case $s=1$, from the definitions \eqref{def1} and \eqref{def2}, and the lemma  \ref{neproj2}, we find
\begin{align*}
{|v|}^{2}_{-1,hk}=(A_{hk}v,v)&=(Av,v)+((A_{hk}-A)v,v)\\
&\leq{|v|}^{2}_{-1}+{||A_{hk}v-Av||}_{L^{2}(\Omega)}{||v||}_{L^2(\Omega)}\\
&\leq{|v|}^{2}_{-1}
+C~h^2k^2\log k\bigg|\log\frac{k}{h}\bigg|{||v||}^{2}_{L^{2}(\Omega)}.
\end{align*}
Let $1\leq s\leq \nu-1$ and assume the result is true up to $s$. 

Now from the definition \eqref{def2} of the discrete negative semi-norm ${|\cdot|}_{-s,hk}$,
\begin{equation}\label{neg1}
{|v|}_{-(s+1),hk}={|A_{hk}v|}_{-(s-1),hk}\leq{|Av|}_{-(s-1),hk}+{|(A_{hk}-A)v|}_{-(s-1),hk}.
\end{equation}
By induction hypothesis
\begin{align}
{|Av|}_{-(s-1),hk}&\leq C\left({|Av|}_{-(s-1)}+h^{s-1}k^{-(s-1)}\log k\bigg|\log\frac{k}{h}\bigg|^{1/2}{||Av||}_{L^{2}(\Omega)}\right)\nonumber\\
&= C\left({|v|}_{-(s+1),hk}+h^{s-1}k^{-(s-1)}\log k\bigg|\log\frac{k}{h}\bigg|^{1/2}{|v|}_{-2}\right).\label{neg2}
\end{align}
Since, for all $j$,  $\lambda_jhk^{-1}>0$ and $$(\lambda_jhk^{-1})^2+(\lambda_jhk^{-1})^{-(s-1)}\geq c>0,$$
which implies
$$c_1(\lambda_jhk^{-1})^2+c_2(\lambda_jhk^{-1})^{-(s-1)}\geq 1$$
for some positive constants $c_1$ and $c_2$.
This further implies
\begin{equation}\label{hh1}
{\lambda_j}^{-2}\leq c_1(hk^{-1})^2+c_2\lambda_j^{-(s+1)}(hk^{-1})^{-(s-1)}.
\end{equation}
Now, using the definition \eqref{spec1} of spectral norm and \eqref{hh1}, we get
\begin{align*}
&{|v|}_{-2}\\
&~~~=\left(\sum^{\infty}_{j=1}\lambda_j^{-2}(v,\phi_j)^2\right)^{\frac{1}{2}}\\
&~~~\leq\left(\sum^{\infty}_{j=1}\bigg(c_1(hk^{-1})^2+c_2\lambda_j^{-(s+1)}(hk^{-1})^{-(s-1)}\bigg)(v,\phi_j)^2\right)^{\frac{1}{2}}\\
&~~~\leq \left(c_1(hk^{-1})^2\sum^{\infty}_{j=1}(v,\phi_j)^2\right)^{\frac{1}{2}}
+\left(c_2(hk^{-1})^{-(s-1)}\sum^{\infty}_{j=1}\lambda_j^{-(s+1)}(v,\phi_j)^2\right)^{\frac{1}{2}}\\
&~~~~=C\bigg(h^2k^{-2}{||v||}_{L^{2}(\Omega)}+h^{-(s-1)}k^{s-1}{|v|}_{-(s+1)}\bigg).
\end{align*}
Then from \eqref{neg2},
\begin{equation}\label{neg3}
{|Av|}_{-(s-1),hk}\leq C\bigg({|v|}_{-(s+1)}+h^{s+1}k^{-(s+1)}\log k\bigg|\log\frac{k}{h}\bigg|^{1/2}{||v||}_{L^{2}(\Omega)}\bigg).
\end{equation}
From induction hypothesis and Lemma \ref{neproj2}, we have
\begin{align}
{|(A-A_{hk})v|}_{-(s-1),hk}
&\leq C\bigg({|(A-A_{hk})v|}_{-(s-1)}\nonumber\\
&~~~~~~~~~+h^{s-1}k^{-(s-1)}\log k\bigg|\log\frac{k}{h}\bigg|^{1/2}{||(A-A_{hk})v||}_{L^{2}(\Omega)}\bigg)\nonumber\\
&\leq C~h^{s+1}k^{-(s+1)}\log k\bigg|\log\frac{k}{h}\bigg|^{1/2}{||v||}_{L^{2}(\Omega)}.\label{neg4}
\end{align}
The first inequality   \eqref{neg5}, follows from \eqref{neg3}, \eqref{neg4} and \eqref{neg1}. Interchanging the role of $A$ and $A_{hk}$,  second inequality follows similarly. 
\end{proof}
Let us define the discrete operator $\Delta_{h,k}:V_{h,k}\rightarrow V_{h,k}$ as follows,
\begin{equation*}
(\Delta_{h,k}u_{h,k},v_{h,k})=-a(u_{h,k},v_{h,k})\hspace{0.2cm}\forall~ v_{h,k}\in V_{h,k}.
\end{equation*}
We note that, $A_{hk}=(-{\Delta}_{h,k})^{-1}$: For 
$$(f_{h,k},v_{h,k})=a(A_{hk}f_{h,k},v_{h,k})=-(\Delta_{h,k}(A_{hk}f_{h,k}),v_{h,k})~\forall~ f_{h,k}\in V_{h,k},$$
which implies $-\Delta_{h,k}(A_{hk}f_{h,k})=f_{h,k}$ for $f_{h,k}\in V_{h,k}$.

Let $\bar{R}_{h,k}f$ be the orthogonal projection of $f$ onto $V_{h,k}$ with respect to $(\cdot,\cdot)$.  Note that $A_{hk}\bar{R}_{hk}=A_{hk}$:  that is for all $v_{h,k}\in V_{h,k}$
$$a((A_{hk}\bar{R}_{hk})f,v_{h,k})=(\bar{R}_{hk}f,v_{h,k})=(f,v_{h,k})=a(A_{hk}f,v_{h,k}).$$
With above notations, the mortar semidiscrete problem can be written as:
\begin{align*}
(\dot{u}_{h,k},v_{h,k})-(\Delta_{h,k} u_{h,k},v_{h,k})&=(\bar{R}_{hk}f,v_{h,k})\hspace{0.3cm}\forall~ v_{h,k}\in V_{h,k},\\
u_{h,k}(0)&=u_{0,h,k}
\end{align*}
that is,
\begin{equation*}
\dot{u}_{h,k}-\Delta_{h,k} u_{h,k}=\bar{R}_{hk}f
\hspace{0.2cm}\text{with}\hspace{0.2cm}
u_{h,k}(0)=u_{0,h,k}.
\end{equation*}
Since $A_{hk}=(-{\Delta}_{h,k})^{-1}$, the above semidiscrete problem can further be written  as  
\begin{equation}\label{22}
A_{hk}\dot{u}_{h,k}+u_{h,k}=A_{hk}\bar{R}_{hk}f=A_{hk}f\hspace{0.2cm}\text{with}\hspace{0.2cm}u_{h,k}(0)=u_{0,h,k}.
\end{equation}
Further, the continuous problem \eqref{h1}-\eqref{h4} can similarly be  written as (cf. \cite{thomee}, page 31) 
\begin{equation}\label{21}
A\dot{u}+u=Af\hspace{0.2cm}\text{with}\hspace{0.2cm}u(0)=u_0.
\end{equation}
Using the above discussions we have the following negative norm estimate
for the mortar solution $u_{h,k}$.
\begin{theorem}
 Let $u$ and $u_{h,k}$ be the solutions of \eqref{h1}-\eqref{h4} and \eqref{h34}-\eqref{ini} respectively. Assume $u_{0}$ and $u_{0,h,k}$ are such that
\begin{equation}\label{neg10}
{|u_{0}-u_{0,h,k}|}_{-s}+{||u_{0}-u_{0,h,k}||}_{L^{2}(\Omega)}
\leq C~h^{s+\eta}k^{-(s+\nu)}\log k\bigg|\log\frac{k}{h}\bigg|^{1/2}{||u_{0}||}_{H^{\nu}(\Omega)},
\end{equation}
where $\eta=\min\{\nu,k+1\}$ and  $0\leq s\leq\nu-2$.

Then for  a geometrically nonconforming partition of $\Omega$ there exists a positive constant $C$ independent of $h$, $k$ and
$u$, such that the following superconvergence  estimate holds for $\nu>2$ and  $0\leq s\leq \nu-2$:
\begin{equation*}
{|u(t)-u_{h,k}(t)|}_{-s}\leq C~h^{\eta+s}k^{-(\nu+s)}\log k\bigg|\log\frac{k}{h}\bigg|^{1/2}\left({||u_{0}||}_{H^{\nu}(\Omega)}
+{||\dot{u}||}_{L^{2}(0,T;H^{\nu}(\Omega))}\right).
\end{equation*}
\end{theorem}
\begin{proof}
Let  $e=u-u_{h,k}$. From \eqref{21} and \eqref{22}, we find
\begin{align*}
A_{hk}\dot{e}+e&=(A_{hk}\dot{u}_{h,k}+u_{h,k})-(A_{hk}\dot{u}+u)\\
&=A_{hk}f-(A\dot{u}+u)+(A-A_{hk})\dot{u}\\
&=(A-A_{hk})(\dot{u}-f)\\
&=(A-A_{hk})A^{-1}u.
\end{align*}
Now, since $A_{hk}=P_{hk}A$, we have
\begin{equation*}
A_{hk}\dot{e}+e=(A-A_{hk})A^{-1}u=u-P_{hk}u=\rho,
\end{equation*}
and
\begin{equation}\label{neg9}
A^{s+1}_{hk}\dot{e}+A^{s}_{hk}e=A^{s}_{hk}\rho.
\end{equation}
Multiplying  \eqref{neg9} with $2\dot{e}$, integrating over $\Omega$ and using definition \eqref{def2}, 
we obtain 
\begin{equation*}
2(A^{s+1}_{hk}\dot{e},\dot{e})+\frac{d}{dt}{|e|}^{2}_{-s,hk}
=2(A^{s}_{hk}\rho,\dot{e})=2\frac{d}{dt}(A^{s}_{hk}\rho,e)-2(A^{s}_{hk}\dot{\rho},e).
\end{equation*}
Further, integrating with respect to $t$ and using the fact that $A^{s+1}_{hk}$ is positive 
semidefinite in $L^2(\Omega)$, that is $(A^{s+1}_{hk}v,v)\geq 0$, we find
\begin{align*}
{|e(t)|}^{2}_{-s,hk}
&\leq {|e(0)|}^{2}_{-s,hk}+2{|\rho(t)|}_{-s,hk}{|e(t)|}_{-s,hk}
+2{|\rho(0)|}_{-s,hk}{|e(0)|}_{-s,hk}\\
&\hspace{0.5cm}+2\int^{t}_{0}{|\rho(s)|}_{-s,hk}{|e(s)|}_{-s,hk}ds\\
&\leq\sup_{s\leq t}{|e(s)|}_{-s,hk}\bigg({|e(0)|}_{-s,hk}+4\sup_{s\leq t}{|\rho(s)|}_{-s,hk}+2\int^{t}_{0}{|\dot{\rho}(s)|}_{-s,hk}ds\bigg).
\end{align*}
Assume $\tau$ such that ${|e(\tau)|}_{-s,hk}=\sup_{s\leq t}{|e(s)|}_{-s,hk}$. Then we get
\begin{align}
{|e(t)|}_{-s,hk}&\leq{|e(\tau)|}_{-s,hk}\nonumber\\
&\leq{|e(0)|}_{-s,hk}+4\sup_{s\leq t}{|\rho(s)|}_{-s,hk}
+2\int^{t}_{0}{|\dot{\rho}(s)|}_{-s,hk}ds\nonumber\\
&\leq{|e(0)|}_{-s,hk}+C\bigg({|\rho(0)|}_{-s,hk}
+\int^{t}_{0}{|\dot{\rho}(s)|}_{-s,hk}ds\bigg).\label{neg8}
\end{align}
From the assumption \eqref{neg10} and Lemma \ref{negeq}, we arrive at
\begin{equation}\label{neg11}
{|e(0)|}_{-s,hk}\leq C~h^{s+\eta}k^{-(s+\nu)}\log k\bigg|\log\frac{k}{h}\bigg|^{1/2}{||u_0||}_{H^{\nu}(\Omega)}.
\end{equation}
Now, from Lemma \ref{neproj2} and Lemma \ref{negeq}, we obtain 
\begin{equation}\label{subha}
{|\rho(s)|}_{-s,hk}\leq C~h^{s+\eta}k^{-(s+\nu)}\log k\bigg|\log\frac{k}{h}\bigg|^{1/2}{||u||}_{H^{\nu}(\Omega)}.
\end{equation}
In particular
\begin{equation}\label{neg12}
{|\rho(0)|}_{-s,hk}\leq C~h^{s+\eta}k^{-(s+\nu)}\log k\bigg|\log\frac{k}{h}\bigg|^{1/2}{||u||}_{H^{\nu}(\Omega)}.
\end{equation}
Similarly, as in \eqref{subha}, we have
\begin{equation}\label{neg13}
{|\dot{\rho}(s)|}_{-s,hk}\leq C~h^{s+\eta}k^{-(s+\nu)}\log k\bigg|\log\frac{k}{h}\bigg|^{1/2}{||\dot{u}||}_{H^{\nu}(\Omega)}.
\end{equation}
Substituting \eqref{neg11}, \eqref{neg12} and \eqref{neg13} in \eqref{neg8}, we get
\begin{equation*}
{|e(t)|}_{-s,hk}\leq C~h^{s+\eta}k^{-(s+\nu)}\log k\bigg|\log\frac{k}{h}\bigg|^{1/2}\bigg({||u_0||}_{H^{\nu}(\Omega)}
+\int^{t}_{0}{||\dot{u}||}_{H^{\nu}(\Omega)}ds\bigg).
\end{equation*}
Finally, from theorem \ref{thm1} and Lemma \ref{negeq} 
\begin{align}
{|e(t)|}_{-s}&\leq C\bigg({|e(t)|}_{-s,hk}+h^{s}k^{-s}\log k\bigg|\log\frac{k}{h}\bigg|^{1/2}{||e(t)||}_{L^{2}(\Omega)}\bigg)\nonumber\\
&\leq C~h^{s+\eta}k^{-(s+\nu)}\log k\bigg|\log\frac{k}{h}\bigg|^{1/2}\bigg({||u_0||}_{H^{\nu}(\Omega)}
+\int^{t}_{0}{||\dot{u}||}_{H^{\nu}(\Omega)}ds\bigg).\nonumber
\end{align}
Hence the theorem follows. 
\end{proof}
\begin{remark}
There are situations where we actually need these negative norm estimates, for instance:
approximation of the integral
$F(u)=\displaystyle\int_{\Omega}uvdx$,
where $v\in H^{\nu-2}(\Omega)$ by $F(u_{h,k})=\displaystyle\int_{\Omega}u_{h,k}v ~dx$, where $u$ and $u_{h,k}$ are the solution of \eqref{h43}-\eqref{h44} and \eqref{h34}-\eqref{ini} respectively,  that is 
\begin{align*}
|F(u)-F(u_{h,k})|&=|(u-u_{h,k},v)|\\
&\leq {|u-u_{h,k}|}_{-(\nu-2)}{|v|}_{\nu-2}\\
&\leq C ~h^{2\eta-2}k^{-(2\nu-2)}\log k\bigg|\log \frac{k}{h}\bigg|\left({||u_{0}||}_{H^{\nu}(\Omega)}
+{||\dot{u}||}_{L^{2}(0,T;H^{\nu}(\Omega))}\right)
\end{align*}
which gives a superconvergent error bound.
\end{remark}
\section{Error estimates for a fully discrete scheme}
Let $r$ be the time step parameter such that $N=T/r$ and $t_{n}=nr$. For a continuous function $\upsilon$ over
$[0,T]$, set the backward difference quotient as:~ $\bar{\partial}\upsilon^{n}=\frac{\upsilon^{n}-\upsilon^{n-1}}{r}$.
The backward Euler approximation is to find a function $U^{n}\in V_{h,k}$ such that
\begin{align}
(\bar{\partial}U^{n},\chi)&+a(U^{n},\chi)=f^{n}(\chi),~n\geq 1\hspace{0.2cm}\forall\chi\in V_{h,k},\label{h22}\\
&U^{0}=u_{0,h,k},\nonumber
\end{align}
where $u_{0,h,k}$ is $I^{h}_{k}u_{0}$ or $P_{h,k}u_{0}$ and $f^{n}(\chi)=\displaystyle\int_{\Omega}f(t_{n})\chi~dx$.
The above problem  can be written in a vector-matrix form
\begin{equation*}
(\tilde{A}+kB)\alpha^{n}=\tilde{A}\alpha^{n-1}+kF(t_{n}),\hspace{0.2cm}n\geq 1,
\end{equation*}
where $\tilde{A}+kB$ is positive definite. Since the matrix $\tilde{A}+kB$ is invertible, the problem \eqref{h22} has a unique solution.
\begin{theorem}
Let $u(t_{n})$ be the solution of \eqref{h13}-\eqref{h14} and $U^{n}\in V_{h,k}$ be an approximation
of $u(t)$ at $t=t_{n}$ given by \eqref{h22}. Then with $u_{0,h,k}=I^{h}_{k}u_{0}$ or $P_{h,k}u_{0}$
and  for  a geometrically nonconforming partition of $\Omega$, there exist positive constants $C$, independent of $h$, $k$ and $r$ such that
\begin{align}
&{||u(t_{n})-U^{n}||}^{2}_{L^{2}(\Omega)}\nonumber\\
&\hspace{0.3cm}\leq C\bigg[h^{2\eta}k^{-2\nu}(\log k)^{2}\bigg|\log\frac{k}{h}\bigg|\sum^{n_{0}}_{i=1}
\bigg({||u_{0}||}^{2}_{H^{\nu}(\Omega_{i})}
&+\int_{0}^{t_{n}}{||\dot{u}||}^{2}_{H^{\nu}(\Omega_{i})}ds\bigg)\nonumber\\
&\hspace{1.2cm}+r^{2}\int_{0}^{t_{n}}{||\ddot{u}||}^{2}_{L^{2}(\Omega)}ds\bigg]\label{h23}
\end{align}
and
\begin{align}
&{||u(t_{n})-U^{n}||}^{2}_{X}\nonumber\\
&\hspace{0.4cm}\leq C\bigg[h^{2(\eta-1)}k^{-2(\nu-1)}\log k\bigg|\log\frac{k}{h}\bigg|\sum^{n_{0}}_{i=1}\bigg({||u_{0}||}^{2}_{H^{\nu}(\Omega_{i})}
+\int_{0}^{t_{n}}{||\dot{u}||}^{2}_{H^{\nu}(\Omega_{i})}ds\bigg)\nonumber\\
&\hspace{1.3cm}+r^{2}\int_{0}^{t_{n}}{||\ddot{u}||}^{2}_{L^{2}(\Omega)}ds\bigg].\label{h24}
\end{align}
\end{theorem}
\begin{proof}
Setting
\begin{equation}
u(t_{n})-U^{n}=\underbrace{u(t_{n})-P_{h,k}u(t_{n})}+\underbrace{P_{h,k}u(t_{n})-U^{n}}=\rho^{n}+\theta^{n}.\nonumber
\end{equation}
We already know the estimates of $\rho^{n}$ from Lemma \ref{proj}, it is enough to estimate $\theta^{n}$ to deduce
the final results. Using \eqref{h13} and \eqref{h22}, we obtain
\begin{equation}\label{h36}
(\bar{\partial}\theta^{n},\chi)+a(\theta^{n},\chi)=(w^{n},\chi)\hspace{0.2cm}\forall \chi\in V_{h,k},
\end{equation}
where
\begin{equation}
w^{n}=\bar{\partial}P_{h,k}u(t_{n})-\dot{u}(t_{n})=-\bar{\partial}\rho^{n}+(\bar{\partial}u(t_{n})-\dot{u}(t_{n}))
=w^{n}_{1}+w^{n}_{2}.\nonumber
\end{equation}
Choosing $\chi=\theta^{n}$ in \eqref{h36}, using coercivity \eqref{h71} of $a(\cdot,\cdot)$, Cauchy-Schwarz
inequality and Young's inequality, we obtain
\begin{align}
\frac{1}{2}\bar{\partial}{||\theta^{n}||}^{2}_{L^{2}(\Omega)}+
c{||\theta^{n}||}^{2}_{X}
&\leq {||w^{n}||}_{L^{2}(\Omega)}{||\theta^{n}||}_{X}\nonumber\\
&\leq \frac{1}{2c}{||w^{n}||}^{2}_{L^{2}(\Omega)}+\frac{c}{2}{||\theta^{n}||}^{2}_{X}\nonumber
\end{align}
since
\begin{align}
(\bar{\partial}\theta^{n},\theta^{n})&=\frac{1}{2}\bar{\partial}{||\theta^{n}||}_{L^{2}(\Omega)}^{2}+
\frac{r}{2}{||\bar{\partial}\theta^{n}||}^{2}_{L^{2}(\Omega)}\nonumber\\
&\geq\frac{1}{2}\bar{\partial}{||\theta^{n}||}^{2}_{L^{2}(\Omega)}.\nonumber
\end{align}
Hence
\begin{equation}
\bar{\partial}{||\theta^{n}||}^{2}_{L^{2}(\Omega)}+{||\theta^{n}||}^{2}_{X}\leq C{||w^{n}||}^{2}_{L^{2}(\Omega)}.\nonumber
\end{equation}
From the definition of $\bar{\partial}$, we get
\begin{equation*}
{||\theta^{n}||}^{2}_{L^{2}(\Omega)}\leq{||\theta^{n-1}||}^{2}_{L^{2}(\Omega)}
+C~r{||w^{n}||}^{2}_{L^{2}(\Omega)}.
\end{equation*}
Repeating this, we arrive at
\begin{equation}\label{h38}
{||\theta^{n}||}^{2}_{L^{2}(\Omega)}\leq{||\theta^{0}||}^{2}_{L^{2}(\Omega)}+
C~r\left(\sum^{n}_{j=1}{||w^{j}_{1}||}^{2}_{L^{2}(\Omega)}+\sum^{n}_{j=1}{||w^{j}_{2}||}^{2}_{L^{2}(\Omega)}\right).
\end{equation}
Again with $u_{0,h,k}=P_{h,k}u_{0}$, $\theta^{0}=0$, otherwise we have $u_{0,h,k}=I^{h}_{k}u_{0}$ and
\begin{equation}\label{h39}
{||\theta^{0}||}_{L^{2}(\Omega)}\leq C~h^{\eta}k^{-\nu}\log k\bigg|\log\frac{k}{h}\bigg|^{1/2}\sum^{n_{0}}_{i=1}{||u_{0}||}_{H^{\nu}(\Omega_{i})}.
\end{equation}
Since
\begin{equation}
w^{j}_{1}=-\frac{(\rho(t_{j})-\rho(t_{j-1}))}{r}=-r^{-1}\int^{t_{j}}_{t_{j-1}}\dot{\rho}(s)ds\nonumber,
\end{equation}
from  Lemma \ref{proj}, we arrive at
\begin{align}\label{h41}
r\sum^{n}_{j=1}{||w^{j}_{1}||}^{2}_{L^{2}(\Omega)}&\leq\sum^{n}_{j=1}
\int^{t_{j}}_{t_{j-1}}{||\dot{\rho}(s)||}^{2}_{L^{2}(\Omega)}ds\nonumber\\
&\leq C~h^{2\eta}k^{-2\nu}(\log k)^{2}\bigg|\log\frac{k}{h}\bigg|\sum^{n_{0}}_{i=1}\int^{t_{n}}_{0}{||\dot{u}||}^{2}_{H^{\nu}(\Omega_{i})}ds.
\end{align}
In order to  estimate $w^{j}_{2}$, apply Taylors series expansion to get
\begin{align}
w^{j}_{2} =\bar{\partial}u(t_{j})-\dot{u}(t_{j})&=r^{-1}(u(t_{j})-u(t_{j-1}))-\dot{u}(t_{j})\nonumber\\
&=-r^{-1}\int^{t_{j}}_{t_{j-1}}(s-t_{j-1})\ddot{u}(s)ds,\nonumber
\end{align}
and hence
\begin{align}
r\sum^{n}_{j=1}{||w^{j}_{2}||}^{2}_{L^{2}(\Omega)}&\leq\sum^{n}_{j=1}\left(\int^{t_{j}}_{t_{j-1}}|s-t_{j-1}|
~{||\ddot{u}||}_{L^{2}(\Omega)}ds\right)^{2}\nonumber\\
&\leq C~r^{2}\int^{t_{n}}_{0}{||\ddot{u}||}^{2}_{L^{2}(\Omega)}ds.\label{h42}
\end{align}
Substituting \eqref{h39}, \eqref{h41} and \eqref{h42} in \eqref{h38}, we find
\begin{align}
{||\theta^{n}||}^{2}_{L^{2}(\Omega)}\leq C~\bigg[&h^{2\eta}k^{-2\nu}(\log k)^{2}\bigg|\log \frac{k}{h}\bigg|\sum^{n_{0}}_{i=1}
\bigg({||u_{0}||}^{2}_{H^{\nu}(\Omega_{i})}
+\int^{t_{n}}_{0}{||\dot{u}||}^{2}_{H^{\nu}(\Omega_{i})}ds\bigg)\nonumber\\
&\hspace{0.5cm}+r^{2}\int^{t_{n}}_{0}{||\ddot{u}||}^{2}_{L^{2}(\Omega)}ds\bigg].\nonumber
\end{align}
Finally,  with the help of  triangle  inequality, lemmas \ref{proj} and \eqref{h45}, we get \eqref{h23}.

In order to find an estimate in $X$-norm, substitute $\chi=\bar{\partial}\theta^{n}$ in \eqref{h36} and
proceed in a similar way to derive \eqref{h24}.
\end{proof}
\begin{remark}
Similarly, for the geometrically conforming partition of the domain, \eqref{h24} becomes
\begin{align*}
{||u(t_{n})-U^{n}||}^{2}_{X}&\leq C\bigg[h^{2(\eta-1)}k^{-2(\nu-1)}\log
 k\sum^{n_{0}}_{i=1}\bigg({||u_{0}||}^{2}_{H^{\nu}(\Omega_{i})}
+\int_{0}^{t_{n}}{||\dot{u}||}^{2}_{H^{\nu}(\Omega_{i})}ds\bigg)\nonumber\\
&\hspace{1cm}+r^{2}\int_{0}^{t_{n}}{||\ddot{u}||}^{2}_{L^{2}(\Omega)}ds\bigg].
\end{align*}
\end{remark}
\section{Numerical experiment}
Consider the problem \eqref{h1}-\eqref{h4}, with exact solution $u=xy(1-x^2)(1-y^2)e^t$ and initial value $u_0=xy(1-x^2)(1-y^2)$ and the coefficient $\alpha=(1,10,10)$. We consider the L-shaped domain $[-1,1]\times[-1,1]\setminus[0,1]\times[0,1]$ (see Figure \ref{fig1}). We conducted the experiment by taking time step parameter $k = O(h^2)$ corresponding to space discretization parameters $h = 1/6, 1/8, 1/10, 1/12, 1/14$. We plot the order of convergence `$p$' of ${||u-u_{h,1}||}_{L^{2}(\Omega)}$  with respect to space
parameter $h$ in the log-log scale, see Figure \ref{fig2}. The  order of convergence `$q$' with respect to time step parameter $r$ depicted in Figure \ref{fig3}. Since
the exact solution is smooth, the convergence rates of the error in $L^2$-norm are obtained as expected, i.e., $O(h^2)$ (with the linear finite elements) and $O(r)$, respectively. We show all computed values in Table \ref{tabbb1} below: 

\begin{table}[hbtp]
\small
\caption{Order of convergence of\hspace{0.2cm}${||u-u_{h,1}||}_{L^{2}(\Omega)}$.}
\centering
\begin{tabular}{lclcl}
\hline
$h$&r&${\|u-u_{h,1}\|}_{L^2(\Omega)}$&p&q\\
\hline
1/6&1/36&0.026451&\\
1/8&1/64& 0.016035& 1.739837756875778&0.869918878437889
\\
1/10&1/100&0.010766&  1.785311789922130&0.892655894961065\\
1/12&1/144&0.0077316&  1.815897138057364
& 0.907948569028682\\
1/14&1/196&0.0058236& 1.838442749983338&0.919221374991669
\\
\hline
\end{tabular}
\label{tabbb1}
\end{table}
\begin{figure}[h]
\includegraphics*[width=12cm,height=9cm]{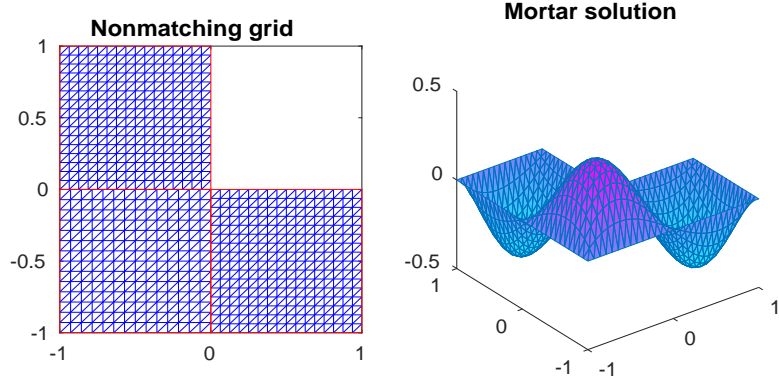}
\centering
\caption{L-shaped domain used with nonmatching grid (left) and a mortar solution (right).}
\label{fig1}
\end{figure}
\begin{figure}[h]
\includegraphics*[width=11cm,height=7cm]{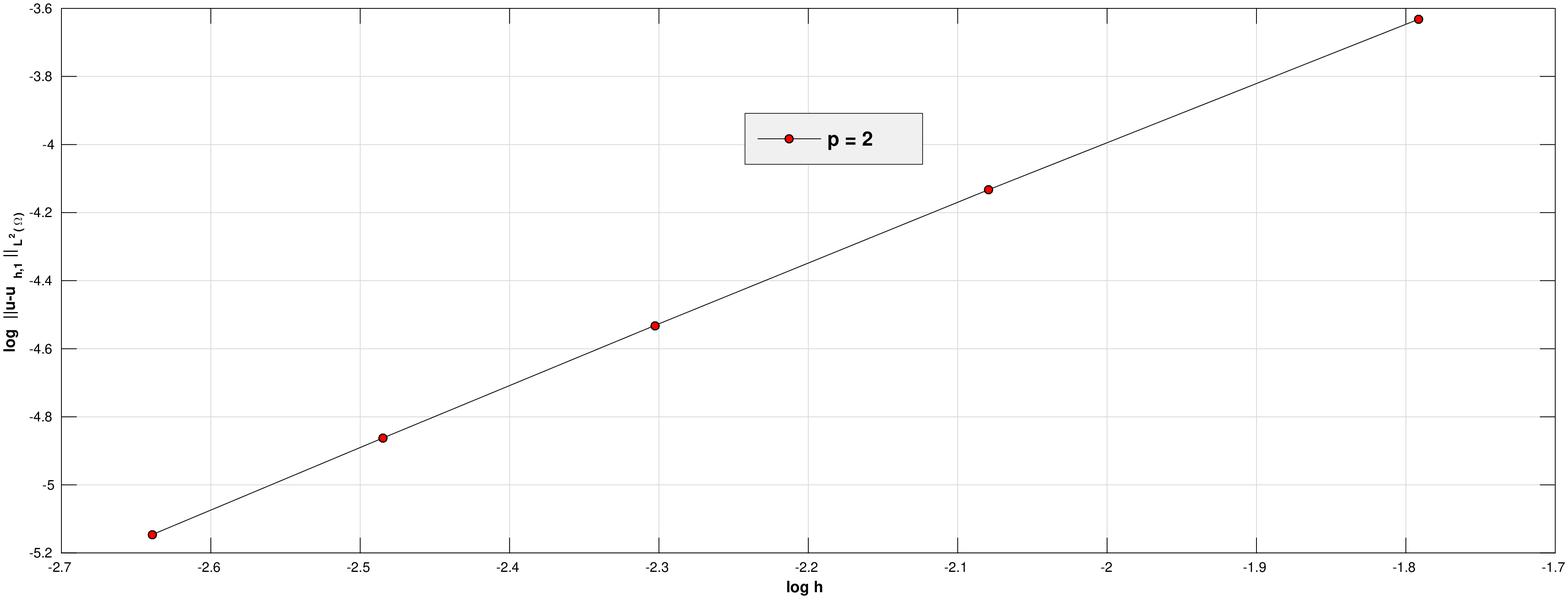}
\centering
\caption{Order of convergence with respect to mesh size $h$.}
\label{fig2}
\end{figure}
\begin{figure}[h]
\includegraphics*[width=11cm,height=7cm]{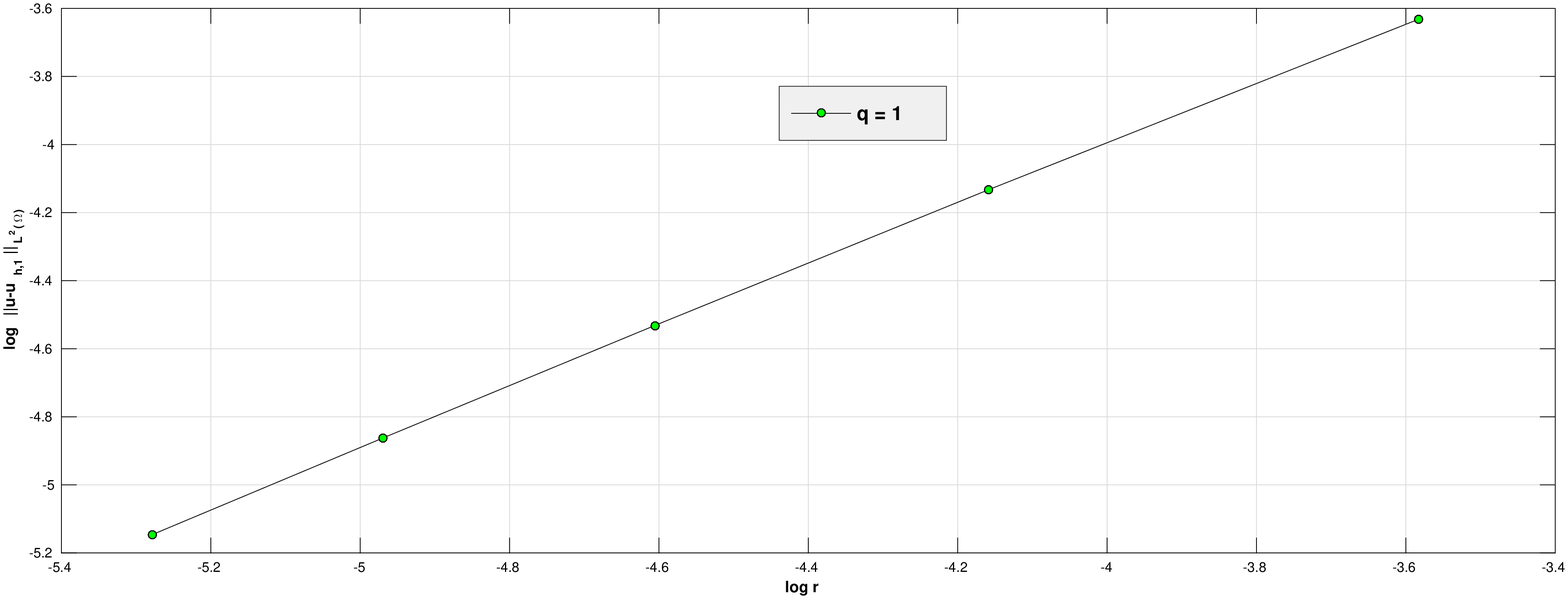}
\centering
\caption{Order of convergence with respect to time step  $r$.}
\label{fig3}
\end{figure}
\section{Conclusion}
We discussed the $hp$ version of the mortar finite element method for a parabolic initial-boundary value problem. Quasioptimal convergence results with a small pollution term $O(\log k)$ are obtained for both semidiscrete and fully discrete methods in both $H^1$- and $L^2$-norms.
With a more regularity assumption superconvergence estimates for the semidiscrete method has been derived in the negative norm.
The fully discrete scheme
is derived using a finite difference method in the temporal direction.
However,  we may derive  the fully
discrete scheme using  finite element methods in temporal direction (cf. \cite{babuska-p}).
Here we considered a problem with homogeneous Dirichlet boundary condition. For nonhomogeneous Dirichlet boundary
condition, we refer to \cite{nonhomogeneous}. 
Although we assumed our domain to be polygonal, one can extend the problem to a domain with curved boundary as in  \cite{curved}.
All the estimates are derived with the help of Sobolev space.
Same results can be expressed using Besov  and Jacobi-weighted Besov spaces (cf. \cite{Guo2m}).

\bibliographystyle{amsplain}

\end{document}